\documentclass[12pt]{amsart}
\usepackage{mathtools}
\usepackage{amsmath, amssymb, mathabx, enumitem}
\usepackage{color}
\usepackage[colorlinks=true,linkcolor=blue,urlcolor=blue,citecolor=blue, pagebackref]{hyperref}
\usepackage[alphabetic]{amsrefs}
\usepackage{tikz-cd}
\usetikzlibrary{arrows}
\usepackage{graphics}
\usepackage{arydshln}

\usepackage[margin=1in]{geometry}

\newcommand{\C}{\mathbb{C}}

\newcommand{\Z}{\mathbb{Z}}

\newcommand{\R}{\mathbb{R}}
\newcommand{\op}{\operatorname}
\newcommand{\wh}{\widehat}

\newtheorem{theorem}{Theorem}[section]

\newtheorem{remark}[theorem]{ Remark}

\newtheorem{corollary}[theorem]{Corollary}

\newtheorem{proposition}[theorem]{Proposition}

\newtheorem{lemma}[theorem]{Lemma}

\newtheorem{definition}[theorem]{Definition}
\newtheorem{definition/lemma}[theorem]{Definition/Lemma}

\title{Extremal Rays of the Equivariant Littlewood-Richardson Cone}
\author{Joshua Kiers}

\begin{document}
\maketitle

\begin{abstract} 
We give
an inductive procedure for finding the extremal rays of the equivariant Littlewood-Richardson cone, which is closely related to the solution space to S.~Friedland's majorized Hermitian eigenvalue problem. In so doing, we solve the ``rational version'' of a problem posed by C.~Robichaux, H.~Yadav, and A.~Yong. 
Our procedure is a natural extension of P.~Belkale's algorithm for the classical Littlewood-Richardson cone. The main tools for accommodating the equivariant setting are certain foundational results of D.~Anderson, E.~Richmond, and A.~Yong. We also study two families of special rays of the cone and make observations about the Hilbert basis of the associated lattice semigroup. 
\end{abstract}

\section{Introduction}

In this work we collide the two worlds of \cite{B2} and \cite{ARY} in order to answer a question of C.~Robichaux, H.~Yadav, and A.~Yong \cite{RYY}. In \cite{B2}, P.~Belkale introduced an algorithm for finding the extremal rays of the Hermitian eigencone (also called the tensor cone or Littlewood-Richardson cone) -- the pointed rational cone which among other things governs the nonvanishing of Littlewood-Richardson coefficients. In \cite{ARY}, D.~Anderson, E.~Richmond, and A.~Yong proved that the \emph{equivariant} Littlewood-Richardson nonvanishing problem is determined by a similar cone, of which the former is a facet, thereby proving the equivariant nonvanishing problem to be saturated. Here, we naturally adapt Belkale's algorithm to the equivariant setting, repeatedly making use of the core Proposition 2.1 from \cite{ARY}, thus finding \emph{most} of the extremal rays of the equivariant Littlewood-Richardson cone. The missing rays are few in number and easily described. 

To provide a little context, a famous problem from the $19^\text{th}$ century is to determine the possible eigenvalues of three Hermitian matrices $A,B$, and $C$ which satisfy $A+B=C$. Horn conjectured \cite{Horn} that a certain recursive set of inequalities on the eigenvalues were necessary and sufficient for such matrices to exist. This turned out to be true, as established by Klyachko \cite{Kly} and Totaro \cite{Tot}.  While Horn's list of constraints is overdetermined, the essential inequalities were found by Belkale \cite{B1} and proven to be minimal by Knutson, Tao, Woodward \cite{KTW}. 
For a much more thorough treatment of this story, see \cite{Fu}. 

There is a natural generalization to this problem. Recall that a Hermitian matrix $A$ \emph{majorizes} another, say $B$, if $A-B$ is positive semidefinite (written $A\ge B$). S.~Friedland \cite{Frie} studied the question: what are the possible eigenvalues of $A$, $B$, and $C$ if $A+B\ge C$? Friedland proposed a set of inequalities, and W.~Fulton \cite{Fu2} showed they correctly answered the problem but were redundant, providing a minimal set that is remarkably the same as the essential Horn inequalities from above! 

In both the aforementioned problems, the sets of integer eigenvalue solutions have special significance. By the Saturation Theorem of \cite{KT}, nonnegative integral solutions to the first problem parametrize the nonzero structure constants (Littlewood-Richardson coefficients) for multiplication in three equivalent settings: Schur polynomials, Grassmannian Schubert classes, and classes of irreducibles in the representation ring of $GL_n$. 
By \cite[Theorem 1.3]{ARY}, the second problem is connected to \emph{double} Schur polynomials and multiplication in the \emph{equivariant} cohomology rings of Grassmannians. 

The solutions to both problems are governed by rational, linear inequalities, so the solution sets, viewed inside the relevant real vector spaces, form convex rational polyhedral cones. These we denote by $\mathsf{LR}$ for the nonvanishing of Littlewood-Richardson problem and $\mathsf{EqLR}$ for the second (equivariant) problem. Such cones have finitely many \emph{extremal rays} -- faces of dimension $1$ -- and the associated semigroups of lattice points have a finite generating set, the \emph{Hilbert basis}, which always includes the set of (primitive points on) extremal rays. 

In \cite[Problem A]{Z}, A.~Zelevinsky posed the question of finding the Hilbert basis for $\mathsf{LR}$ explicitly, and Robichaux, Yadav, and Yong \cite[\S 6]{RYY} asked the same question for $\mathsf{EqLR}$. While both questions remain open, Belkale's algorithm \cite{B2} serves as a partial (``rational'') solution to Zelevinsky's question in that it gives formulas for finding the extremal rays of $\mathsf{LR}$. 
Our main result is an analogous solution to Robichaux, Yadav, and Yong's question, which we obtain by 
extending Belkale's algorithm to the equivariant setting. 

After reviewing the various cones and their inequalities (\S \ref{notes}), as well as recalling Belkale's algorithm in detail (\S \ref{belk}), we show
that for any Horn facet of the equivariant Littlewood-Richardson cone, our algorithm produces the extremal rays on that facet (Theorems \ref{protheo}, \ref{indhat}). Unlike in \cite{B2}, there are some rays not on any Horn facet, and we explicitly enumerate these (Theorem \ref{extras}). Finally, we show by example (\S \ref{x}) that there do in general ($r\ge6$) exist Hilbert basis elements which do not lie on an extremal ray. This phenomenon has not been observed for the classical Littlewood-Richardson cones. 

\subsection{Acknowlegements} The author thanks Alex Yong for calling attention to this problem and for encouragement along the way, as well as Prakash Belkale and Dave Anderson for helpful discussions and feedback. 

\section{Notation and inequalities for the cones}\label{notes}

Both of the eigenvalue problems above can be considered with a greater number of summands. Let $r\ge 1$ and fix some $s\ge 3$ ($s-1$ will be the number of summands). Let us consider the space of possible eigenvalues of $r\times r$ matrices $A_1,\hdots,A_{s-1}$, and $C$ such that $\sum A_j = C$ or $\sum A_j\ge C$. We will always list the eigenvalues of a matrix in decreasing order, counted with multiplicity, for example: $\lambda = (\lambda_1\ge \lambda_2\ge \hdots \ge \lambda_r)\in \R^r$. Define 
$$
\mathsf{C}_r^s = \left\{(\lambda^1,\hdots,\lambda^{s-1},\nu)\in (\R^r)^{s} \left| \begin{array}{c} \exists \text{ $r\times r$ Herm. matrices $A_j$, $C$ with eigenvalues $\lambda^j,\nu$} \\ \text{s.t. $A_1+\hdots+A_{s-1} = C$} \end{array}\right.\right\}.
$$
Similarly define 
$$
\mathsf{EqC}_r^s = \left\{(\lambda^1,\hdots,\lambda^{s-1},\nu)\in (\R^r)^{s} \left| \begin{array}{c} \exists \text{ $r\times r$ Herm. matrices $A_j$, $C$ with eigenvalues $\lambda^j,\nu$} \\ \text{s.t. $A_1+\hdots+A_{s-1} \ge C$} \end{array}\right.\right\}.
$$

\subsection{Inequalities describing the cones}

To describe the Horn inequalities, which cut out the above cones $\mathsf{C}_r^s, \mathsf{EqC}_r^s$ inside $(\R^r)^s$, we first introduce the notation for Grassmannian Schubert calculus which we will need. Let $n$ be a sufficiently large integer, to be specified as needed, and let $\op{Gr}(r,\C^n)$ denote the Grassmannian of $r$-dimensional subspaces in $\C^n$. The cohomology ring $H^*(\op{Gr}(r,\C^n))$ has a distinguished graded basis given by classes of Schubert varieties $[X_\lambda]$, one for each partition $\lambda = (\lambda_1\ge\hdots \ge \lambda_r)\in \Z_{\ge 0}^r$ such that $\lambda_1\le n-r$ (i.e., the Young diagram for $\lambda$ fits inside an $r\times (n-r)$ rectangle); here 
$$
X_\lambda = \{V\in \op{Gr}(r,\C^n)| \dim (V\cap F_{n-r+i-\lambda_i})\ge i,~ \forall 1\le i\le r\},
$$ 
and $F_\bullet$ is the fixed flag on the standard basis of $\C^n$ defined by $F_i = \C e_n\oplus \hdots \oplus \C e_{n-i+1}$.
The complex codimension of $X_\lambda$ is $|\lambda|:=\lambda_1+\hdots+\lambda_r$. 

We define the (multiple-factor) structure coefficients $c_{\lambda^1,\hdots,\lambda^{s-1}}^\nu$ by the rule
$$
[X_{\lambda^1}]\cdots [X_{\lambda^{s-1}}] = \sum_{\nu} c_{\lambda^1,\hdots,\lambda^{s-1}}^\nu [X_\nu] \in H^*(\op{Gr}(r,\C^n))
$$
for $n$ larger than $r+\sum_{j=1}^{s-1} \lambda_1^j$ (these classes and products are \emph{stable}: it does not matter which $n$ we take, as long as it is big enough.) When $s=3$, these coefficients are the Littlewood-Richardson numbers. 

Partitions $\lambda$ whose Young diagrams fit in an $r\times (n-r)$ rectangle are in bijection with $r$-element subsets of $[n]:=\{1,\hdots,n\}$. Such a subset we will typically write in increasing order as $I = \{i_1<i_2<\hdots<i_r\}$. The bijection is often denoted $\tau$ and goes as follows:
~\\
\begin{center}
\begin{tikzcd}
I = \{i_1<i_2<\hdots<i_r\} \arrow[r,mapsto,"\tau"]&  \tau(I) = (i_r-r \ge \hdots i_2-2 \ge i_1-1)
\end{tikzcd}
\end{center}
~\\
\noindent
Note that the definition is independent of $n$. If no confusion is likely to arise, we'll just write $c_{I_1,\hdots,I_{s-1}}^{K}$ instead of $c_{\tau(I_1),\hdots,\tau(I_{s-1})}^{\tau(K)}$. 

Now we can state the celebrated theorem resolving Horn's conjecture: 

\begin{theorem}[\cite{Kly,B1,KTW}]
Let $\lambda^j = (\lambda_1^j, \hdots , \lambda_r^j)$, $j=1,\hdots,s-1$ and $\nu = (\nu_1, \hdots ,\nu_r)$ be $s$ sequences of $r$ real numbers each. Then $(\lambda^1,\hdots,\lambda^{s-1},\nu)\in \mathsf{C}_r^s$ if and only if 
\begin{enumerate}
\item[(i)] $\lambda_1^j\ge \hdots \ge \lambda_r^j$ for each $j=1,\hdots,s-1$; and $\nu_1\ge \hdots \ge \nu_r$ (since we write eigenvalues in nonincreasing order);
\item[(ii)] $\sum_{j=1}^{s-1} |\lambda^j| = |\nu|$; and 
\item [(iii)] for every $1\le d<r$ and every collection $I_1,\hdots,I_{s-1},K$ of $d$-element subsets of $[r]$ satisfying 
\begin{align}\label{lr=1}
c_{\tau(I_1),\hdots,\tau(I_{s-1})}^{\tau(K)} = 1,
\end{align}
the inequality 
$$
\sum_{j=1}^{s-1} \sum_{a\in I_j} \lambda^j_a \ge \sum_{k\in K} \nu_k
$$
holds. 
\end{enumerate}
\end{theorem}

Concerning Friedland's problem, we have the following result of W.~Fulton:

\begin{theorem}[\cite{Fu2}]
Let $\lambda^j = (\lambda_1^j, \hdots , \lambda_r^j)$, $j=1,\hdots,s-1$ and $\nu = (\nu_1, \hdots ,\nu_r)$ be $s$ sequences of $r$ real numbers each. Then $(\lambda^1,\hdots,\lambda^{s-1},\nu)\in \mathsf{EqC}_r^s$ if and only if they satisfy (i) and (iii) above as well as 
\begin{enumerate}
\item[(ii')] $\sum_{j=1}^{s-1} |\lambda^j| \ge |\nu|$\footnote{in fact, this is the unique ``Horn inequality'' for $d=r$.}.
\end{enumerate}
\end{theorem}

\subsection{The Littlewood-Richardson cones}

Amazingly, Horn's inequalities also provide the solutions to the following problems in Schubert calculus. One can ask for which tuples $(\lambda^1,\hdots,\lambda^{s-1},\nu)$ the Littlewood-Richardson number $c_{\lambda^1,\hdots,\lambda^{s-1}}^\nu$ is nonzero, and the answer is in fact the same: 

\begin{theorem}[\cite{Kly,KT}]\label{lr}
Let $\lambda^j = (\lambda_1^j, \hdots , \lambda_r^j)$, $j=1,\hdots,s-1$ and $\nu = (\nu_1, \hdots ,\nu_r)$ be $s$ sequences of $r$ nonnegative integers each. Then $c_{\lambda^1,\hdots,\lambda^{s-1}}^\nu\ne 0$ if and only if (i), (ii), and (iii) hold above. 
\end{theorem}

This has the following equivariant analogue: $\op{Gr}(r,\C^n)$ has an action of $T = (\C^*)^n$, and $T$ fixes each Schubert variety. Moreover, the equivariant classes $[X_\lambda]$ once again form a basis for $H^*_T(\op{Gr}(r,\C^n))$, as a module over $\Z[t_1,\hdots,t_n] = H^*_T(\{pt\})$. We therefore can define structure ``coefficients'' (polynomials in the variables $t_i$) $C_{\lambda^1,\hdots,\lambda^{s-1}}^\nu$ by the rule 
$$
[X_{\lambda^1}]\cdots [X_{\lambda^{s-1}}] = \sum_{\nu} C_{\lambda^1,\hdots,\lambda^{s-1}}^\nu [X_\nu] \in H^*_T(\op{Gr}(r,\C^n)),
$$
once again assuming $n$ is large enough. The nonvanishing question for $C_{\lambda^1,\hdots,\lambda^{s-1}}^\nu$ was settled by D.~Anderson, E.~Richmond, and A.~Yong: 

\begin{theorem}[\cite{ARY}\footnote{technically, they cover the case $s=3$. By induction, one obtains the statement for all $s>3$.}]\label{elr}
Let $\lambda^j = (\lambda_1^j, \hdots , \lambda_r^j)$, $j=1,\hdots,s-1$ and $\nu = (\nu_1, \hdots ,\nu_r)$ be $s$ sequences of $r$ nonnegative integers each. Then $C_{\lambda^1,\hdots,\lambda^{s-1}}^\nu\ne 0$ if and only if (i), (ii'), and (iii) hold above, as well as 
\begin{enumerate}
\item[(iv)] $\nu_i\ge \lambda_i^j$ for every $i\in [r]$ and $j\in [s-1]$. 
\end{enumerate}
\end{theorem}

Criterion (iv) is also written $\lambda^j\subseteq \nu$ (for every $j\in [s-1]$); i.e., the Young diagram for $\lambda^j$ fits inside the Young diagram for $\nu$. 

Let us therefore define two more cones, $\mathsf{LR}_r^s\subset \mathsf{C}_r^s$ and $\mathsf{EqLR}_r^s\subset \mathsf{EqC}_r^s$, as follows. 
Set
$$
\mathsf{LR}_r^s = \{(\lambda^1,\hdots,\lambda^{s-1},\nu)\in \mathsf{C}_r^s | \lambda^j_r\ge 0 \text{ for every } j \in [s-1]\}\footnote{the reader may notice we have omitted the requirement $\nu_r\ge 0$; this is because it follows from $\lambda^j_r\ge 0$, $\sum |\lambda^j| = |\nu|$, and one of the Horn inequalities.}
$$
and 
$$
\mathsf{EqLR}_r^s = \{(\lambda^1,\hdots,\lambda^{s-1},\nu)\in \mathsf{EqC}_r^s | \lambda^j_r\ge 0 \text{ and }\lambda^j\subseteq \nu \text{ for every } j \in [s-1]\}.
$$

Then Theorems \ref{lr} and \ref{elr} say that the partitions yielding nonvanishing (equivariant) structure constants are exactly the sets of lattice points $\mathsf{LR}_r^s\cap \Z^{rs}$ and $\mathsf{EqLR}_r^s\cap \Z^{rs}$, respectively. (On an important historical note, Klyachko showed that $(\lambda^1,\hdots,\lambda^{s-1},\nu)\in \mathsf{LR}_r^s\cap \Z^{rs}\iff c_{N\lambda^1,\hdots,N\lambda^{s-1}}^{N\nu}\ne 0$ for some $N>0$; the Saturation Theorem of Knutson and Tao resolved the conjecture that $c_{N\lambda^1,\hdots,N\lambda^{s-1}}^{N\nu}\ne 0\iff c_{\lambda^1,\hdots,\lambda^{s-1}}^{\nu}\ne 0$, thus proving Theorem \ref{lr}.)

\section{Belkale's algorithm for the rays of $\mathsf{LR}_r^s$}\label{belk}

In \cite{B2}, P.~Belkale considers the following rational cone: 
\begin{align*}
\Gamma_r(s)&:=\{(\lambda^1,\hdots,\lambda^{s-1},\lambda^s) | (\lambda^1,\hdots,\lambda^{s-1},-w_0\lambda^s) \in \mathsf{C}_r^s\text{ and each } |\lambda^j|=0\},  
\end{align*}
which parametrizes the possible eigenvalues for \emph{traceless} $r\times r$ Hermitian matrices $A_1,\hdots,A_s$ such that $A_1+\hdots+A_s = 0$. Here $w_0$ is the involutive permutation 
$$
w_0(\lambda_1,\hdots,\lambda_r) = (\lambda_r,\hdots,\lambda_1).
$$
Belkale describes the extremal rays that lie on an arbitrary \emph{Horn facet}, i.e., the face of $\Gamma_r(s)$ where one of the Horn inequalities (iii) holds with equality. In order to recall that description in our present notation, we introduce the following cone: 
$$
\mathsf{C}_{SL_r}^s:=\{(\lambda^1,\hdots,\lambda^{s-1},\nu)\in \mathsf{C}_r^s | \lambda^j_r= 0 \text{ for every } j \in [s-1]\},
$$
which is isomorphic to $\Gamma_r(s)$ via the Killing form isomorphism, after twisting the last entry by $-w_0$. For any $j\in [r]$, let $\omega_j$ denote the partition $(\underbrace{1,\hdots,1}_{j},\underbrace{0,\hdots,0}_{r-j})$. Then the linear map 
\begin{align*}
C_{SL_r}^s&\to \Gamma_r(s)\\
(\lambda^1,\hdots,\lambda^{s-1},\nu) &\mapsto \left(\lambda^1-\frac{|\lambda^1|}{r}\omega_r,\hdots, \lambda^{s-1}-\frac{|\lambda^{s-1}|}{r}\omega_r,-w_0\nu + \frac{|\nu|}{r}\omega_r\right)
\end{align*}
is well-defined and an isomorphism of rational cones with inverse 
$$
(\lambda^1,\hdots,\lambda^{s-1},\lambda^s) \mapsto \left(\lambda^1-\lambda^1_r\omega_r, \hdots, \lambda^{s-1}_r-\lambda^{s-1}\omega_r,-w_0\lambda^s - \sum_{j=1}^{s-1} \lambda_r^j \omega_r\right).
$$

\subsection{Relationships between the cones}

So far we have accumulated several related cones, which fit in this diagram:
$$
\begin{array}{ccccc}
\mathsf{C}_{SL_r}^s & \subset & \mathsf{LR}_r^s & \subset & \mathsf{C}_r^s\\
 & & \rotatebox{-90}{$\subset$} & & \rotatebox{-90}{$\subset$} \\ 
 & & & & \\
 & & \mathsf{EqLR}_r^s & \subset & \mathsf{EqC}_r^s
\end{array}
$$
(Though one could define a sixth cone that ``completes'' the lower left corner of the diagram, that cone does not appear to be helpful for our purposes.) 

The two vertical inclusions, as well as the left-most one, are inclusions of faces -- that is, the smaller cone is defined inside the bigger one by taking one or more linear inequality valid on the bigger cone and forcing it/them to hold with equality. Moreover, we have the following structural results.  

\begin{proposition}
For each $j\in [s-1]$, let $x_j = (0,\hdots,\underbrace{\omega_r}_{j^\text{th} \text{ position}},0,\hdots,0,\omega_r)\in \R^{rs}$. 

We have the following internal decompositions, the first of them direct: 
\begin{align*}
\mathsf{LR}_r^s &= \mathsf{C}_{SL_r}^s \oplus \bigoplus_{j=1}^{s-1} \R_{\ge 0}x_j\\
\mathsf{C}_r^s&= \mathsf{LR}_r^s + \bigoplus_{j=1}^{s-1} \R x_j\\
\mathsf{EqC}_r^s&= \mathsf{EqLR}_r^s + \bigoplus_{j=1}^{s-1} \R x_j
\end{align*}
\end{proposition}

We postpone the straightforward proof until Section \ref{odds}. As a consequence, we get relationships among extremal rays: for example, the extremal rays of $\mathsf{LR}_r^s$ are the rays of $\mathsf{C}_{SL_r}^s$ together with $\{x_j|1\le j\le s-1\}$. Since $\mathsf{C}_r^s$ and $\mathsf{EqC}_r^s$ are not pointed cones (they contain the linear subspaces $\R x_j$), their sets of extremal rays are not well-defined. Nonetheless, they are generated over $\R_{\ge0}$ by the extremal rays of $\mathsf{LR}_r^s$ and $\mathsf{EqLR}_r^s$ (respectively) and $\{\pm x_j|1\le j\le s-1\}$. 

While Belkale's algorithm was developed for the Horn facets of $\mathsf{C}_{SL_r}^s$, it applies equally well to the Horn facets of $\mathsf{LR}_r^s$, as each $x_j$ belongs to every Horn facet and is in the image of every induction map. 

\subsection{The algorithm}

\begin{definition}
Suppose $I_1,\hdots,I_{s-1},K$ satisfy (\ref{lr=1}). The associated Horn facet $\mathcal{F}_{I_1,\hdots,I_{s-1}}^{K}$ is 
$$
\mathcal{F}_{I_1,\hdots,I_{s-1}}^{K}:= \left\{(\lambda^1,\hdots,\lambda^{s-1},\nu)| \sum_{j=1}^{s-1}\sum_{a\in I_j}\lambda_a^j = \sum_{k\in K}\nu_k\right\}\cap \mathsf{LR}_r^s.
$$
We will often write $\mathcal{F}$ for short, if the $d$-elements subsets of $[r]$ are understood. 
\end{definition}

Fix such a collection $I_1,\hdots,I_{s-1},K$ satisfying (\ref{lr=1}). The rays on $\mathcal{F}$ come in two types. The first are obtained as follows. Fix any $j\in [s-1]$ and an $a\in I_j$ such that $a+1\not\in I_j$, and $a<r$. Alternatively, let $j=s$ and find an $a\in K$ such that $a-1\not\in K$, but $a>1$. The pair $(j,a)$ is a ``type I ray datum.'' Modify just $I_j$ by swapping $a$ for $a+1$ (or $a$ for $a-1$ in case $j=s$) to produce a new collection $I_1',\hdots,I_{s-1}',K'$. 
From the choice of $j,a$, we get a ``type I'' ray $r(j,a) = (\lambda^1,\hdots,\lambda^{s-1},\nu)$ by this procedural definition: 
\begin{enumerate}
\item First, set $\lambda^k_r = 0$ for every $k\in [s-1]$. 
\item For every $k\in [s-1]$, if $b>1$ satisfies $b\in I_k'$ and $b-1\not\in I_k'$, set $I_k'' = I_k' \cup\{b-1\}\setminus \{b\}$ and $I_\ell'' = I_\ell'$ for $\ell\ne k$, $K'' = K'$. Then set 
$$
\lambda^k_{b-1} -\lambda^k_b = c_{I_1'',\hdots,I_{s-1}''}^{K''}.
$$
If $b>1$ satisfies $b\not\in I_k'$ or $b-1\in I_k'$, then set $\lambda^k_{b-1} - \lambda^k_b = 0$. Thus we have determined $\lambda^1,\hdots,\lambda^{s-1}$ completely. 
\item If $c<r$ satisfies $c\in K'$ and $c+1\not\in K'$, set $I_k'' = I_k'$ for all $k$ and $K'' = K'\cup\{c+1\}\setminus\{c\}$. Then set 
$$
\nu_c-\nu_{c+1}= c_{I_1'',\hdots,I_{s-1}''}^{K''}.
$$
If $c<r$ satisfies $c\not\in K'$ or $c+1\in K'$, then set $\nu_c-\nu_{c+1} = 0$. 
\item All that remains is to find $\nu_r$. This can be done by using the requirement 
$$
|\lambda^1|+\hdots+|\lambda^{s-1}|= |\nu|.
$$
However, there is an alternative rule. If $r\not\in K'$, then $\nu_r=0$; otherwise, set $K'' = K'\cup\{r+1\}\setminus\{r\}$, and $I_k'' = I_k'$ for all $k$. Then 
$$
\nu_r = c_{I_1'',\hdots,I_{s-1}''}^{K''},
$$
where technically this Littlewood-Richardson coefficient is calculated in the cohomology of a bigger Grassmannian ($\op{Gr}(d,\C^{r+1})$ will suffice). 
\end{enumerate}

\begin{remark}
The author is indebted to G.~Orelowitz for observing the alternative rule in step (4), which allows one to treat (4) as a new case of (3), defining ``$\nu_{r+1}=0$''. 
\end{remark}

\begin{theorem}[Belkale]\label{typeI}
Every $r(j,a)$ produced by the preceding procedure generates an extremal ray of $\mathsf{C}_{SL_r}^s\subset \mathsf{LR}_r^s$ which lies on the face $\mathcal{F}$. They form a linearly independent set, enumerated $\{r_1,\hdots,r_q\}$, say, and they span a simplicial subcone $\mathcal{F}_1 = \prod_{i=1}^q\R_{\ge0}r_i \subseteq \mathcal{F}$. 
\end{theorem}

In fact, linear independence follows from \cite{B2}*{Lemma 4.2}: 

\begin{lemma}\label{ortho}
Say $r(j,a) = (\lambda^1,\hdots,\lambda^{s-1},\nu)$ and $r(\wh j,\hat a) = (\hat \lambda^1,\hdots,\hat \lambda^{s-1},\hat \nu)$ are two distinct type I rays on $\mathcal{F}$. (For simplicity assume $j,\wh j<s$; analogous statements hold if $j=s$ or $\wh j=s$ or both.) Then $\lambda_a^j-\lambda_{a+1}^j=1$ and $\lambda_{\hat a}^{\wh j}-\lambda_{\hat a+1}^{\wh j} = 0$; likewise $\hat \lambda_{\hat a}^{\wh j}-\hat \lambda_{\hat a+1}^{\wh j}=1$ and $\hat \lambda_{a}^j - \hat \lambda_{a+1}^j=0$.
\end{lemma}

This leads one to define the subcone 
$$
\mathcal{F}_2:= \left\{(\lambda^1,\hdots,\lambda^{s-1},\nu) \in \mathcal{F} \left|\begin{array}{c} \lambda^j_a-\lambda^j_{a+1}=0 \text{ for every type I datum }(j,a),j<s; \\ 
\nu_{a-1}-\nu_a =0 \text{ for every type I datum }(s,a)
\end{array}\right.\right\}.
$$
Clearly the addition map 
$$
\mathcal{F}_1\times \mathcal{F}_2 \to \mathcal{F}
$$
is an injection of rational cones, given the above lemma and the definition of $\mathcal{F}_2$. In \cite[Proposition 4.3]{B2}, we find that this map is also surjective. Therefore the remaining extremal rays of $\mathcal{F}$ are just the extremal rays of $\mathcal{F}_2$. For this, we have a surjective linear map onto $\mathcal{F}_2$ from a cone whose rays we can expect to know inductively. 

We'll define that map momentarily. First, for a $d$-element subset $I$ of $[r]$ and a partition $\lambda$, we define
$$
\lambda_I = (\lambda_{i_1}\ge \lambda_{i_2}\ge \hdots\ge\lambda_{i_d}).
$$
We also write $\bar I$ for the complement of $I$ inside $[r]$. 
Next, define the invertible map 
\begin{align*}
\pi:\R^{rs} &\to \R^{ds}\times \R^{(r-d)s}\\
(\lambda^1,\hdots,\lambda^{s-1},\nu)&\mapsto (\lambda^1_{I_1},\hdots,\lambda^{s-1}_{I_{s-1}},\nu_K),(\lambda^1_{\bar I_1},\hdots,\lambda^{s-1}_{\bar I_{s-1}},\nu_{\bar K}).
\end{align*}

The proof of the following result can be found in \cite{Fu}*{Proposition 8} and is also a consequence of the factorization rule of \cite{KTT}. 

\begin{proposition}\label{facto}The above $\pi$ restricts to a map 
$$
\pi:\mathcal{F} \to \mathsf{LR}_d^s\times \mathsf{LR}_{r-d}^s.
$$
\end{proposition}
Although this restriction of $\pi$ is not necessarily surjective, we do clearly have 
$$
\pi^{-1} (\mathsf{LR}_d^s\times \mathsf{LR}_{r-d}^s) \subseteq \R\mathcal{F}.
$$
Finally let 
$$
p_2:\R\mathcal{F} \to \R\mathcal{F}_2
$$
be the projection onto the second factor. Then define the induction map $\op{Ind}$ to be $p_2\circ\pi^{-1}$. 

\begin{theorem}[Belkale]\label{typeII}
The linear map 
$$
\op{Ind} = p_2\circ \pi^{-1}: \mathsf{LR}_d^s \times \mathsf{LR}_{r-d}^s \to \mathcal{F}_2
$$
is well-defined and surjective. 
\end{theorem}

The surjectivity means that every extremal ray of $\mathcal{F}_2$ is the image of an extremal ray of $\mathsf{LR}_d^s\times \mathsf{LR}_{r-d}^s$. Moreover, the extremal rays of $\mathsf{LR}_d^s\times \mathsf{LR}_{r-d}^s$ are all of the form $a\times 0$ or  $0\times b$ where $a$ (resp., $b$) is an extremal ray of $\mathsf{LR}_d^s$ (resp., $\mathsf{LR}_{r-d}^s$). As both $d$ and $r-d$ are strictly smaller than $r$, we get an inductive algorithm for finding the extremal rays of $\mathsf{LR}_r^s$, starting with $\mathsf{LR}_1^s$. 

\begin{remark}
Actually, Belkale's $\op{Ind}$ has a proper subspace of $\mathsf{LR}_d^s\times \mathsf{LR}_{r-d}^s$ as its domain, but it nonetheless surjects onto $\mathcal{F}_2\cap \mathsf{C}_{SL_r}^s$. With the larger domain comes a larger kernel (see Corollary \ref{noway}) but a more transparent generalization to $\mathsf{EqLR}_r^s$. 
\end{remark}

\section{Adaptation of the algorithm for $\mathsf{EqLR}_r^s$}

Theorems \ref{typeI} and \ref{typeII} can be straightforwardly adapted to find the extremal rays of the Horn facets of $\mathsf{EqLR}_r^s$. Once again assume that $I_1,\hdots,I_{s-1},K$ satisfy (\ref{lr=1}). Define
$$
\hat{\mathcal{F}} = \left\{(\lambda^1,\hdots,\lambda^{s-1},\nu)\left| ~\sum_{j=1}^{s-1}\sum_{a\in I_j}\lambda_a^j = \sum_{k\in K}\nu_k\right.\right\}\cap \mathsf{EqLR}_r^s.
$$
Note that $\mathcal{F} = \hat{\mathcal{F}}\cap \mathsf{LR}_r^s$, and since $\mathcal{F}$ is a facet of $\hat{\mathcal{F}}$, every extremal ray of $\mathcal{F}$ is a ray of $\hat{\mathcal{F}}$ as well. 

Therefore type I rays on $\mathcal{F}$ are naturally considered type I rays on $\hat{\mathcal{F}}$; i.e., Theorem \ref{typeI} holds verbatim with $\hat{\mathcal{F}}$ instead of $\mathcal{F}$. It is the set of type II rays which will increase from $\mathcal{F}$ to $\hat{\mathcal{F}}$. 
Define 
$$
\hat{\mathcal{F}}_2:=\left\{(\lambda^1,\hdots,\lambda^{s-1},\nu) \in \hat{\mathcal{F}} \left|\begin{array}{c} \lambda^j_a-\lambda^j_{a+1}=0 \text{ for every type I datum }(j,a),j<s; \\ 
\nu_{a-1}-\nu_{a} =0 \text{ for every type I datum }(s,a)
\end{array}\right.\right\}.
$$
Extremal rays of $\hat{\mathcal{F}}_2$ will be called type II rays for $\hat{\mathcal{F}}$. 

Before we come to the proof of the decomposition $\hat{\mathcal{F}} = \mathcal{F}_1\times \hat{\mathcal{F}_2}$, let us recall an important consequence of \cite[Proposition 2.1(B)]{ARY} (they state it for $s=3$ and for integer vectors, but the statement below follows easily). For a pair of vectors $\lambda,\mu$ we use $\lambda\subseteq \mu$ to mean $\lambda_i\le \mu_i$ for every $i\in [r]$. 

\begin{proposition}[Anderson-Richmond-Yong]\label{ARYp21B}
Suppose $(\lambda^1,\hdots,\lambda^{s-1},\nu)\in \mathsf{EqLR}_r^s$. Then for any $j\in [s-1]$, one can find a $\lambda^{j,\downarrow}$ such that
\begin{enumerate}
\item $\lambda^{j,\downarrow}\subseteq \lambda^j$,
\item $(\lambda^1,\hdots,\lambda^{j,\downarrow},\hdots,\lambda^{s-1},\nu)\in \mathsf{EqLR}_r^s$, and
\item $\sum_{k\ne j} |\lambda^k| + |\lambda^{j,\downarrow}| = |\nu|.$
\end{enumerate}
\end{proposition}

In other words, every element of $\mathsf{EqLR}_r^s$ possesses several ``shadows'' in $\mathsf{LR}_r^s$ obtainable by shrinking any single one of the first $s-1$ partitions. 

\begin{theorem}\label{protheo}
The addition map 
$$
\mathcal{F}_1\times \hat{\mathcal{F}}_2\to \hat{\mathcal{F}}
$$
is an isomorphism of cones. 
\end{theorem}

\begin{proof}
Once again, this map is clearly injective given Lemma \ref{ortho} and the definition of $\hat{\mathcal{F}}_2$. Now let us show it is surjective. 

Let $x = (\lambda^1,\hdots,\lambda^{s-1},\nu)\in \hat{\mathcal{F}}$ be arbitrary. If $x$ already belongs to $\hat{\mathcal{F}}_2$, then we are done. Otherwise, there exists some type I datum $(j,a)$ such that $\lambda^j_a-\lambda^j_{a+1}>0$ (in case $j<s$) or $\nu_{a-1}-\nu_{a}>0$ (in case $j=s$). Set $\beta = \lambda^j_a-\lambda^j_{a+1}$ or $\beta =\nu_{a-1}-\nu_a$, depending on the case. We will show that 
\begin{align}\label{shrink}
x - \beta r(j,a)\in \hat{\mathcal{F}};
\end{align}
then by induction on the number of such type I data, we can assume $x-\beta r(j,a)\in \mathcal{F}_1\times \hat{\mathcal{F}}_2$, from which the result follows. In fact, to show (\ref{shrink}) it suffices to show $x-\beta r(j,a)\in \mathsf{EqLR}_r^s$, since their difference clearly lies in the hyperplane spanned by $\hat{\mathcal{F}}$. 

Toward that end, we begin with a simple observation. Write 
$
r(j,a) = (\bar \lambda^1,\hdots,\bar\lambda^{s-1},\bar \nu).
$
Fix any $\ell\in [s-1]$ and let $\tilde x_\ell = (\lambda^1,\hdots,\lambda^{\ell,\downarrow},\hdots,\lambda^{s-1},\nu)$ be an element of $\mathsf{LR}_r^s$ as guaranteed by Proposition \ref{ARYp21B}. Set $\mu^k = \lambda^k$ for $k\ne \ell$ and $\mu^\ell = \lambda^{\ell,\downarrow}$, so in this notation 
$$
\tilde x_\ell = (\mu^1,\hdots,\mu^{s-1},\nu).
$$
Observe that $\tilde x_\ell$ is automatically on $\hat{\mathcal{F}}$ since $x$ is, and 
$$
\sum_{k=1}^{s-1}\sum_{b\in I_k} \lambda^k_b \ge \sum_{k=1}^{s-1}\sum_{b\in I_k} \mu^k_b  \ge \sum_{k\in K}\nu_k.
$$
Moreover, this shows us that $\lambda^\ell_a = \mu^\ell_a$ whenever $a\in I_\ell$. This means that $\mu^j_a-\mu^j_{a+1}\ge \beta$ (equality holds if $\ell\ne j$, but if $\ell=j$ then $\mu^\ell_{a+1}$ could be smaller while $\mu^\ell_a$ is unchanged.)

So $\tilde x_\ell$ still has $\mu^j_a-\mu^j_{a+1}$ (or $\nu_{a-1}-\nu_a$) at least equal to $\beta$. By \cite[Proposition 4.3]{B2}, we know that 
$$
\tilde x_\ell - \beta r(j,a)\in \mathcal{F}\subset \hat{\mathcal{F}},
$$
and this holds for every $\ell\in [s-1]$. 

Finally, we show that $x-\beta r(j,a)\in \mathsf{EqLR}_r^s$ by verifying each of inequalities (i) (along with nonnegativity), (ii'), (iii), and (iv). 

\begin{enumerate}
\item[(i)] Let $k\in [s-1]$. Then $\lambda^k = \mu^k$ from $\tilde x_\ell$ as long as $\ell\ne k$, so choose some such $\ell$. Because $\tilde x_\ell - \beta r(j,a)$ belongs to $\mathsf{LR}_r^s$, $\mu^k-\beta \bar\lambda^k$ is a partition; therefore $\lambda^k-\beta \bar \lambda^k$ is a partition. 
For any $\ell$, $\nu - \beta \bar \nu$ is a partition for the same reason. 
\item[(ii')] This inequality follows by subtracting $\beta\sum_{k=1}^{s-1} |\bar\lambda^k| = \beta |\bar \nu|$ from $\sum_{k=1}^{s-1} |\lambda^k| \ge |\nu|$.
\item[(iii)] Let $(J_1,\hdots,J_{s-1},L)$ parametrize a Horn inequality (i.e., $c_{J_1,\hdots,J_{s-1}}^{L} = 1$). Fix an arbitrary $\ell\in [s-1]$. Then since $\tilde x_\ell - \beta r(j,a)$ satisfies this Horn inequality, we have 
$$
\sum_{k=1}^{s-1}\sum_{b\in J_k} \lambda^k_b - \beta \bar\lambda^k_b \ge \sum_{k=1}^{s-1}\sum_{b\in J_k} \mu^k_b - \beta \bar\lambda^k_b  \ge \sum_{k\in L}\nu_k.
$$
\item[(iv)] Let $k\in [s-1]$ and choose $\ell\ne k$. Since $\tilde x_\ell - \beta r(j,a)\in \mathsf{LR}_r^s$, we know that 
$$
\mu^k - \beta \bar\lambda^k \subseteq \nu-\beta \bar\nu.
$$
Since $\lambda^k = \mu^k$, the needed inequalities follow. 
\end{enumerate}\end{proof}

\begin{remark}
The proof shows that the addition map is surjective even on lattice points. 
\end{remark}

So just like in the classical setting, to find the extremal rays of $\hat{\mathcal{F}}$ it now suffices to find the extremal rays of $\hat{\mathcal{F}}_2$ (the type II rays). Recall the definition of $\pi$ from earlier: 
\begin{align*}
\pi:\R^{rs} &\to \R^{ds}\times \R^{(r-d)s}\\
(\lambda^1,\hdots,\lambda^{s-1},\nu)&\mapsto (\lambda^1_{I_1},\hdots,\lambda^{s-1}_{I_{s-1}},\nu_K),(\lambda^1_{\bar I_1},\hdots,\lambda^{s-1}_{\bar I_{s-1}},\nu_{\bar K}).
\end{align*}
Akin to Proposition \ref{facto}, we have the following generalization in the equivariant setting. 
\begin{proposition}\label{TBC}The map $\pi$ restricts to 
$$
\pi:\hat{\mathcal{F}} \to \mathsf{LR}_d^s\times \mathsf{EqLR}_{r-d}^s.
$$
\end{proposition}
The proof of this fact can be found in \cite{Fu2}*{Claim, pg. 30}. 
We also provide a short argument based on Proposition \ref{facto} and \cite[Proposition 2.1]{ARY}. 
\begin{proof}
Suppose $(\lambda^1,\hdots,\lambda^{s-1},\nu)\in \mathsf{EqLR}_r^s$ satisfies the Horn inequality given by $(I_1,\hdots,I_{s-1},K)$ with equality. 
By Proposition \ref{ARYp21B} we know that we can find an $\tilde x = (\lambda^1,\hdots,\lambda^{j,\downarrow},\hdots,\lambda^{s-1},\nu)\in \mathsf{LR}_r^s$. Moreover, we showed in the previous proof that $\tilde x\in \mathcal{F}$, and that $\lambda^{j,\downarrow}_{I_j} = \lambda^j_{I_j}$. Applying Proposition \ref{facto}, we have 
$$
\pi(\tilde x) = (\lambda^1_{I_1},\hdots,\lambda^{s-1}_{I_{s-1}},\nu_K)\times (\lambda^1_{\bar I_1},\hdots,\lambda^{j,\downarrow}_{\bar I_j}, \hdots, \lambda^{s-1}_{\bar I_{s-1}},\nu_{\bar K}) \in \mathsf{LR}_d^s\times \mathsf{LR}_{r-d}^s.
$$
Since $(\lambda^1_{\bar I_1},\hdots,\lambda^{j,\downarrow}_{\bar I_j}, \hdots, \lambda^{s-1}_{\bar I_{s-1}},\nu_{\bar K})\in \mathsf{LR}_{r-d}^s$, and since $\lambda^j_{\bar I_j}\subseteq \nu_{\bar K}$ (by \cite[Lemma 2.7]{ARY}), we can apply \cite[Proposition 2.1(A)]{ARY} to conclude that 
$
(\lambda^1_{\bar I_1},\hdots, \lambda^{s-1}_{\bar I_{s-1}},\nu_{\bar K})\in \mathsf{EqLR}_{r-d}^s.
$
\end{proof}

Once again, even though $\pi$ is not surjective, it is true that $\pi^{-1}(\mathsf{LR}_d^s\times \mathsf{EqLR}_{r-d}^s)\subseteq \R \hat{\mathcal{F}}$. Let 
$$
\hat p_2 : \R\hat{\mathcal{F}}\to \R\hat{\mathcal{F}}_2
$$
be the second projection, and define $\wh{\op{Ind}}= \hat p_2 \circ\pi^{-1}$. Note that, if $\iota$ denotes the inclusion map $\mathcal{F}_2\subset \hat{\mathcal{F}}_2$, we have $\wh{\op{Ind}}\big|_{\mathsf{LR}_d^s\times \mathsf{LR}_{r-d}^s} = \iota\circ \op{Ind}$.

\begin{theorem}\label{indhat}
The linear map 
$$
\wh{\op{Ind}} = \hat p_2\circ\pi^{-1}: \mathsf{LR}_d^s\times \mathsf{EqLR}_{r-d}^s\to \hat{\mathcal{F}}_2
$$
is well-defined and surjective.
\end{theorem}

In order to save some on notation in what follows, we will use $v^{(k)}$ to denote the $k^\text{th}$ entry (itself a vector of length $r$) of $v$; i.e., $v = (v^{(1)},v^{(2)},\hdots,v^{(s)})\in \R^{rs}$.

\begin{proof}
Linearity is obvious and surjectivity follows since Proposition \ref{TBC} implies $\wh{\op{Ind}}$ has a section given by $\pi$.

It remains to show that the image of $\wh{\op{Ind}}$ is contained in $\hat{\mathcal{F}}_2$. 
Let $x=(\mu^1,\hdots,\mu^{s-1},\kappa)\in \mathsf{LR}_d^s$ and $y=(\alpha^1,\hdots,\alpha^{s-1},\gamma)\in \mathsf{EqLR}_{r-d}^s$. Write $\pi^{-1}(x,y)$ as $(\lambda^1,\hdots,\lambda^{s-1},\nu)$, but remember that $\lambda^1,\hdots,\lambda^{s-1},\nu$ may not be partitions. 

For each $\ell\in [s-1]$, find a $y_\ell = (\beta^1,\hdots,\beta^{s-1},\gamma)$ with $\beta^\ell = \alpha^{\ell,\downarrow}$, $\beta^k = \alpha^k$ for $k\ne \ell$ as in the conclusion of Proposition \ref{ARYp21B}. 
Let $\mathfrak{J}$ be the set of all the type I pairs $(j,a)$. 

Write 
$$
\pi^{-1} (x,y) = \sum_{\mathfrak{J}} c_{j,a} r(j,a) + z
$$
where $c_{j,a}\in \R$ and $z=\wh{\op{Ind}}(x,y) \in \R\hat{\mathcal{F}}_2$. Likewise, express 
$$
\pi^{-1}(x,y_\ell) = \sum_{\mathfrak{J}} d^\ell_{j,a} r(j,a)+z_\ell.
$$

We claim that for every $(j,a)\in \mathfrak{J}$, $d^\ell_{j,a}\ge c_{j,a}$. 

First examine the case $j<s$. Then $a\in I_j$ and $a+1\in \bar I_j$. Therefore $\lambda^j_a = \mu_p$ for some $p$ and $\lambda^j_{a+1} = \alpha_q$ for some $q$, and (by Lemma \ref{ortho} and the definition of $\hat{\mathcal{F}}_2$) $c_{j,a} = \lambda^j_a-\lambda^j_{a+1} = \mu_p - \alpha_q$. Likewise, $d^\ell_{j,a} = \mu_p-\beta_q$. Since $\beta_q\le \alpha_q$, we get $d^\ell_{j,a}\ge c_{j,a}$. 

Second, for $j=s$, we have $a\in K$ and $a-1\in \bar K$. So $\nu_{a-1} = \gamma_q$ for some $q$ and $\nu_a = \kappa_p$ for some $p$, and $c_{j,a} = \gamma_q-\kappa_p$. But since the last coordinates of $y$ and $y_\ell$ agree, we also have $d^\ell_{j,a} = \gamma_q-\kappa_p$, hence $c_{j,a} = d^\ell_{j,a}$. 

Now we verify that $z\in \mathsf{EqLR}_r^s$ by verifying the inequalities (i) (plus nonnegativity), (ii'), (iii), and (iv). 
\begin{enumerate}
\item[(i)] 
If $k\ne \ell$, then $\pi^{-1}(x,y_\ell)^{(k)} = \pi^{-1}(x,y)^{(k)}$. Therefore 
\begin{align}\label{critical}
z^{(k)} = z_\ell^{(k)} + \sum_{\mathfrak{J}}(d^\ell_{j,a}-c_{j,a}) r(j,a)^{(k)}.
\end{align}
It follows at once that each $z^{(k)}$ is a partition (recall that the above claim holds for arbitrary $\ell$). 
\item[(ii')] We have $\sum_{j=1}^{s-1} |z^{(j)}| - |z^{(s)}| = \sum_{j=1}^{s-1} |\mu^j| - |\kappa| + \sum_{j=1}^{s-1} |\alpha^j| - |\gamma| =  \sum_{j=1}^{s-1} |\alpha^j| - |\gamma| \ge 0$. 
\item[(iii)] Fix any $\ell\in [s-1]$. Even though $\pi^{-1}(x,y)^{(\ell)}$ and $\pi^{-1}(x,y_\ell)^{(\ell)}$ are not likely to be partitions, they do still satisfy the entrywise bound
$$
\pi^{-1}(x,y_\ell)^{(\ell)} \subseteq \pi^{-1}(x,y)^{(\ell)}.
$$
Therefore $z^{(\ell)}\supseteq z_\ell^{(\ell)} + \sum_{\mathfrak{J}} (d_{j,a}^\ell-c_{j,a}) r(j,a)^{(\ell)}$. Since $z_\ell=\op{Ind}(x,y_\ell)$ satisfies all the Horn inequalities, as of course do the $r(j,a)$, and since every $d_{j,a}^\ell-c_{j,a}\ge 0$, $z$ must satisfy the Horn inequalities as well. 
\item[(iv)] If $k_0\ne \ell$, then since $z_\ell^{(k_0)}\subseteq z_\ell^{(s)}$ and each $r(j,a)^{(k_0)}\subseteq r(j,a)^{(s)}$, we get from (\ref{critical}) applied to both $k=k_0$ and $k=s$ that $z^{(k_0)}\subseteq z^{(s)}$. 
\end{enumerate}

We have just shown that $z = \wh{\op{Ind}}(x,y)$ belongs to $\mathsf{EqLR}_r^s$. Since $z\in \R\hat{\mathcal{F}}_2$, we get that $z\in \hat{\mathcal{F}}_2$ as desired. 
\end{proof}

It is helpful in practice to have a formula for $\wh{\op{Ind}}$, or really for $\wh p_2$, so we record that here for use in \S \ref{x}. 
\begin{lemma}
Once again let $\mathfrak{J} = \{(j,a)\}$ be the collection of type I data on the facet $\hat{\mathcal{F}}$, with associated rays $r(j,a)$. 
The map $\wh p_2:\R\hat{\mathcal{F}} \to \R\hat{\mathcal{F}}_2$ sends $(\lambda^1,\hdots,\lambda^{s-1},\nu)$ to 
\begin{align}\label{tIIform}
 (\lambda^1,\hdots,\lambda^{s-1},\nu) 
- \sum_{\tiny \begin{array}{c}(j,a)\in \mathfrak{J} \\ j<s\end{array}} (\lambda^j_a-\lambda^j_{a+1})r(j,a) 
- \sum_{\tiny \begin{array}{c}(j,a)\in \mathfrak{J} \\ j=s\end{array}} (\nu_{a-1}-\nu_{a})r(j,a). 
\end{align}
\end{lemma}

It is possible, in general, for $\wh{\op{Ind}}$ to take extremal rays to non-extremal rays, or even to $0$. In fact, by \cite{B2}*{Proposition 9.3}, we have a good understanding of the kernel of $\wh{\op{Ind}}$. 

\begin{corollary}\label{noway}
The kernel of $\wh{\op{Ind}}$ is spanned by the elements $\pi(r(j,a))$ as $(j,a)$ ranges over $\mathfrak{J}$. Moreover, the number of extremal rays of $\mathsf{LR}_d^s\times \mathsf{EqLR}_{r-d}^s$ which map to $0$ under $\wh{\op{Ind}}$ is equal to $|\mathfrak{J}|$, and these rays therefore also form a basis of $\ker\wh{\op{Ind}}$.  
\end{corollary} 

\begin{proof}
Clearly each $\pi(r(j,a))$ is in the kernel. Since $\pi$ is an invertible map, the collection $\{\pi(r(j,a))| (j,a)\in \mathfrak{J}\}$ is a linearly independent set. If it has the cardinality of $\dim \ker \wh{\op{Ind}}$ we will have shown that they are a basis. For this we observe that 
\begin{align*}
\dim \ker \wh{\op{Ind}} &= \dim(\R \mathsf{LR}_r^s\times \R\mathsf{EqLR}_r^s) - \dim \R\hat{\mathcal{F}}_2\\
&= \dim \R \hat{\mathcal{F}} - \dim \R\hat{\mathcal{F}}_2\\
&= \dim \R \mathcal{F}_1 = |\mathfrak{J}|.
\end{align*}
Now, the only rays of $\mathsf{LR}_d^s\times \mathsf{EqLR}_{r-d}^s$ which map to $0$ are in fact rays of $\mathsf{LR}_d^s\times \mathsf{LR}_{r-d}^s$, since $\wh{\op{Ind}}$ preserves the extent to which inequality (ii') is strict.

From \cite{B2}*{Proposition 9.3(3), Corollary 9.4}, there are exactly $|\mathfrak{J}|-(s-1)$ rays of $\mathsf{C}_{SL_d}^s\times \mathsf{C}_{SL_{r-d}}^s$ which map to $0$. 
Moreover $x_j\times 0$ maps to $0$ for every $j\in [s-1]$, as follows from \cite{BKiers}*{Corollary 60}. 
For dimension reasons, the elements $0\times x_j$ cannot map to $0$; otherwise $\ker \wh{\op{Ind}}$ would have dimension greater than $|\mathfrak{J}|$. 
\end{proof}

\section{Special rays}
\subsection{Rays of $\mathsf{EqLR}_r^s$ not on any Horn facet}

In this section, we find the exceptional rays that are not on any Horn facet. We begin with a couple of lemmas on certain rays. 

\begin{lemma}
Suppose $x=(\omega_{k_1},\omega_{k_2},\hdots,\omega_{k_{s-1}},\omega_\ell)$ belongs to $\mathsf{EqLR}_r^s$. Then $\R_{\ge0}x$ is an extremal ray. 
\end{lemma}

\begin{proof}
By hypothesis, the inequalities (iv) imply that each $k_i\le \ell$. 
The only way to write $x$ as a sum of partitions is 
$$
(\omega_{k_1},\hdots,\omega_{k_{s-1}},\omega_\ell) = (q_1\omega_{k_1},\hdots,q_{s-1}\omega_{k_{s-1}},q_s\omega_\ell)+(r_1\omega_{k_1},\hdots,r_{s-1}\omega_{k_{s-1}},r_s\omega_\ell),
$$
where $q_i+r_i=1$ for each $i$. If both elements on the RHS belong to $\mathsf{EqLR}_r^s$, then we have both $q_s\ge q_i$ and $r_s\ge r_i$ for each $i$ by (iv), which forces equalities to hold since the two sides each add to $1$. So the summands are parallel. 
\end{proof}

\begin{lemma}\label{dominants}
The triple $(\omega_{k_1},\hdots,\omega_{k_{s-1}},\omega_\ell)\in \mathsf{EqLR}_r^s$ if and only if the inequalities 
\begin{align*}
~\forall i\in [s-1],~~k_i&\le \ell \text{ and }\\
\sum_{i=1}^{s-1} k_i&\ge \ell
\end{align*}
are satisfied.
\end{lemma}

\begin{proof}
($\Rightarrow$) These are just inequalities (iv) and (ii'). 

($\Leftarrow$) Observe that $(\omega_\ell,0,\hdots,0,\omega_\ell)\in \mathsf{EqLR}_r^s$ since the corresponding LR coefficient is $1$. By \cite{ARY}*{Proposition 2.1(A)}, we also have $(\omega_\ell,\omega_{k_2},\omega_{k_3},\hdots,\omega_{k_{s-1}},\omega_\ell)\in \mathsf{EqLR}_r^s$. 

By \cite{ARY}*{Proposition 2.1(B)}, for any value $t$ between $\ell-\sum_{i\ge 2} k_i=|\omega_\ell|-\sum_{i\ge 2}|\omega_{k_i}|$ and $\ell=|\omega_\ell|$ we may find $\lambda$ such that $(\lambda,\omega_{k_2},\hdots,\omega_\ell)\in \mathsf{EqLR}_r^s$, $\lambda\subset \omega_\ell$, and $|\lambda|=t$. By assumption, $\ell-\sum_{i\ge 2} k_i\le k_1\le \ell$, so we may find such a $\lambda$ with $|\lambda|=k_1$. Of course, the only partition of $k_1$ fitting inside $\omega_\ell$ is $\omega_{k_1}$, so $(\omega_{k_1},\hdots,\omega_{k_{s-1}},\omega_\ell)\in \mathsf{EqLR}_r^s$ as promised. 
\end{proof}

\begin{theorem}\label{extras}~

\begin{enumerate}
\item[(I)] Suppose $(\lambda^1,\hdots,\lambda^{s-1}, \nu)$ gives an extremal ray of $\mathsf{EqLR}_r^s$ satisfying 
\begin{enumerate}
\item[(A)] $\sum_{j=1}^{s-1}|\lambda^j|>|\nu|$;
\item[(B)] each inequality of (iii) holds strictly. 
\end{enumerate}
Then $\lambda^1=\hdots = \lambda^{s-1}=\nu$ and this common partition is $\omega_\ell$ for some $\ell$. 
\item[(II)] Furthermore, every element of the form $(\omega_\ell,\omega_\ell,\hdots,\omega_\ell)$ is an extremal ray of $\mathsf{EqLR}_r^s$, and such a ray lies on no Horn facet if and only if $\ell\ge r/(s-1)$. 
\end{enumerate}
\end{theorem}

\begin{proof}[Proof of (I)]
Assume for the sake of contradiction that $\nu_j>\lambda^1_j$ for some $j$, and assume $j$ is as small as possible. Then since $\lambda^1_{j-1} = \nu_{j-1}\ge \nu_j>\lambda^1_j$, we know that $\lambda^1\pm \epsilon(\underbrace{1,\hdots,1}_{j-1},0,\hdots,0)$ is still a partition for $\epsilon>0$ a small enough real number. 

Let $k$ be the greatest index satisfying $k\ge j$ and $\nu_{k}=\nu_j$, and for each $p=2,\hdots,s-1$, let $i_p$ be the smallest index satisfying $i_p\le k$ and $\lambda^p_{i_p+1}=\lambda^p_{k+1}$ (where $\lambda^p_{r+1}=0$ by definition). Thus in particular $\lambda^p_{i_p}>\lambda^p_{i_p+1}$, so both $\nu\pm\epsilon(\underbrace{1,\hdots,1}_{k},0,\hdots,0)$ and $\lambda^p\pm\epsilon(\underbrace{1,\hdots,1}_{i_p},0,\hdots,0)$ are also partitions for small $\epsilon$. 

Set 
$$
z_{\pm\epsilon}:=(\lambda^1,\hdots,\lambda^{s-1},\nu)\pm\epsilon\left((\underbrace{1,\hdots,1}_{j-1},0,\hdots,0),(\underbrace{1,\hdots,1}_{i_2},0,\hdots,0),\hdots, (\underbrace{1,\hdots,1}_{k},0,\hdots,0)\right).
$$
If we show that for small enough $\epsilon$ both $z_{+\epsilon}$ and $z_{-\epsilon}$ belong to $\mathsf{EqLR}_r^s$, we will have a contradiction regarding the extremality of $(\lambda^1,\hdots,\lambda^{s-1},\nu)$. Since $(\lambda^1,\hdots,\lambda^{s-1},\nu)$ satisfies (ii') and (iii) strictly, we are guaranteed that $z_{\pm\epsilon}$ satisfy (ii') and (iii) for small enough $\epsilon$. The preceding paragraph showed that $z_{\pm\epsilon}$ satisfy (i) (and nonnegativity) for small $\epsilon$ as well. So it suffices to show they satisfy (iv) for small enough $\epsilon$: 

\begin{enumerate}
\item[$\bullet$] We first show $z_{\pm\epsilon}^{(1)} \subseteq z_{\pm\epsilon}^{(s)}$.
If $i\le j$, the inequality $\lambda_i\pm\epsilon\le\nu_i\pm\epsilon$ is satisfied. For $k\ge i>j$, we have $\lambda_i<\nu_j=\nu_i$, so $\lambda_i<\nu_i\pm\epsilon$ for small enough $\epsilon$. For $i>k$, the inequality $\lambda_i\le\nu_i$ is unchanged. 

\item[$\bullet$] Now let $p\in \{2,\hdots,s-1\}$; we'll verify that $z_{\pm\epsilon}^{(p)} \subseteq z_{\pm\epsilon}^{(s)}$. 
If $i\le i_2$, the inequality $\lambda^p_i\pm\epsilon\le \nu_i\pm\epsilon$ is satisfied. For $i_2<i\le k$, we have 
$\lambda^p_i=\lambda^p_{k+1}\le \nu_{k+1}<\nu_{k}\le \nu_i$, so $\lambda^p_i<\nu_i\pm\epsilon$ for small enough $\epsilon$. For $i>k$, the inequality $\mu_i\le \nu_i$ is unchanged. 
\end{enumerate}
So we have a contradiction and it must instead be true that $\lambda^1=\nu$. The same argument applies to any $\lambda^j$. 

If our common partition $\lambda^1=\hdots=\lambda^{s-1}=\nu$ is expressed $\sum c_i\omega_i$ in the $\{\omega_i\}$ basis, note that for any $c_i\ne 0$, $(\lambda^1,\hdots,\lambda^{s-1},\nu)\pm\epsilon(\omega_i,\hdots \omega_i,\omega_i)$ satisfies (ii') and (iii) for small enough $\epsilon$ and maintains (iv) with equalities, so there must be only one such nonzero coefficient. 

Each $(\omega_\ell,\hdots,\omega_\ell,\omega_\ell)$ is an extremal ray by Lemma \ref{dominants}.  
\end{proof}
\begin{proof}[Proof of (II)]

Finally, we show that $\ell< r/(s-1)\iff (\omega_\ell,\hdots,\omega_\ell,\omega_\ell)$ lies on a Horn facet. 

First, suppose $\ell<r/(s-1)$. Find integers $a_1,\hdots,a_{s-1}$ such that\footnote{unless $r-1$ is divisible by $s-1$, choices abound here. Say $r-1\equiv b$ mod $s-1$, where $0\le b<s-1$. Then one can take $a_i = (r-1-b)/(s-1)$ for $i\ge 2$ and $a_1 = (r-1+b(s-2))/(s-1)$.}
\begin{align*}
r-1 &= a_1+\hdots+a_{s-1},\\
\text{ each } a_j&\ge \left\lfloor \frac{r-1}{s-1}\right\rfloor.
\end{align*}
Set $I_p = \{a_p+1\}$ and $K = \{r\}$. Then $c_{\tau(I_1),\hdots,\tau(I_{s-1})}^{\tau(K)} = 1$. Since 
$$\ell<r/(s-1)\le \left\lfloor \frac{r-1}{s-1} \right\rfloor+1\le a_p+1$$
for each $p$, 
the associated Horn inequality, applied to $(\omega_\ell,\hdots,\omega_\ell,\omega_\ell)$, is $0+\hdots+0\ge 0$, thus satisfied with equality.

Second, suppose $\ell\ge r/(s-1)$. Assume that $(\omega_\ell,\hdots,\omega_\ell,\omega_\ell)$ lies on a Horn facet associated to $d$-element subsets $I_1,\hdots,I_{s-1},K$. So
$$
\sum_{a\in I_1,a\le \ell} 1 + \hdots + \sum_{a\in I_{s-1},a\le \ell} 1 = \sum_{k\in K,k\le \ell} 1.
$$
Since $c_{\omega_\ell,0,\hdots,0}^{\omega_\ell}=1$, it must also be true that 
$$
\sum_{a\in I_1,a\le \ell} 1 + 0 \ge \sum_{k\in K,k\le \ell} 1,
$$
in which case the above must hold with equality and the sets $\{a\in I_p|a\le \ell\}$ must be empty for $p\ge 2$. By symmetry, the set $\{a\in I_1:a\le \ell\}$ is also empty, so every $I_p$ consists only of elements $>\ell$. 

Now, the stipulation $|\tau(I_1)|+\hdots +|\tau(I_{s-1})| = |\tau(K)|$ forces 
$$
\sum_{p=1}^{s-1}\sum_{a\in I_p}a = \sum_{k\in K}k+(s-2)d(d+1)/2. 
$$
However, a lower bound for the LHS (i.e., each $I_p$ is as small as possible at $\{\ell+1,\hdots,\ell+d\}$) is $(s-1)d\ell+(s-1)d(d+1)/2$, while an upper bound for the RHS (i.e., where $K = \{r-d+1,\hdots,r-d+d\}$) is $d(r-d)+(s-1)d(d+1)/2$. Therefore we get $(s-1)d\ell\le dr-d^2$. Assuming $r\le (s-1)\ell$, this forces
$$
(s-1)d\ell+d^2\le rd \le (s-1)d\ell,
$$
an obvious contradiction to $d>0$.
\end{proof}

\subsection{Rays of $\mathsf{EqLR}_r^s$ on every Horn facet}\label{odds} 

In contrast, there are some extremal rays of $\mathsf{EqLR}_r^s$ that lie on every Horn facet. If $r=1$, then there are no inequalities (iii), so to make the current discussion more uniform, we will treat (ii') as a Horn inequality, at least if $r=1$. 

\begin{proposition}
Suppose that $(\lambda^1,\hdots,\lambda^{s-1},\nu)$ belongs to $\mathsf{EqLR}_r^s$ and lies on every Horn facet (so really belongs to $\mathsf{LR}_r^s$). Then each $\lambda^p$ is a scalar multiple of $\omega_r$, as is $\nu$. 
\end{proposition}

\begin{proof}
Actually, we will need surprisingly few of the Horn inequalities. We begin with $\sum_{p=1}^{s-1} \lambda^p_1 = \nu_1$. Now for any $1<k\le r$, there is also the equality
$$
\lambda^1_k+\sum_{p=2}^{s-1} \lambda^p_1 = \nu_k.
$$
Combining these two equations, we get $\lambda^1_1-\lambda^1_k = \nu_1-\nu_k$. The choice of $\lambda^1$ was arbitrary, so we get 
$$
\lambda^1_1-\lambda^1_k = \hdots = \lambda^{s-1}_1-\lambda^{s-1}_k = \nu_1-\nu_k.
$$
Letting $a$ represent that common difference, we wish to show $a=0$. Consider the quantities 
\begin{align*}
0&=\sum_{p=1}^{s-1} \lambda^p_1 - \nu_1,\\
b&=\sum_{p=1}^{s-1} \sum_{i=1}^{k-1}\lambda^p_i - \sum_{i=1}^{k-1}\nu_i,\\
c&=\sum_{p=1}^{s-1} \sum_{i=1}^{k}\lambda^p_i - \sum_{i=1}^{k}\nu_i.
\end{align*}
Of course $b$ and $c$ are also $0$ by assumption, but even if we had not assumed that these Horn inequalities held with equality, we could use the following argument to show $a=c=0$, given that (inducting on $k$) $b = 0$, because
$$
0 +b - c = \underbrace{a+\hdots+a}_{s-1} - a
$$
and at the very least $c\ge0, a\ge 0$. 

Therefore all partitions $\lambda^p$ and $\nu$ are multiples of $\omega_r$. Furthermore, even if we had not assumed $\sum |\lambda^p| = |\nu|$, we would have proved it along the way, except in the single case $r=1$. 
\end{proof}

\begin{corollary}
The extremal rays of $\mathsf{EqLR}_r^s$ lying on every Horn facet are spanned by the collection 
$$
x_j = (0,\hdots,0,\underbrace{\omega_r}_j,0,\hdots,\omega_r), ~j\in [s-1].
$$
\end{corollary}

\begin{proposition}
Consider the addition map $f:\bigoplus \R_{\ge0}x_j \oplus\mathsf{C}_{SL_r}^s\to \mathsf{LR}_r^s$. We claim $f$ is an additive, $\R_{\ge 0}$-linear bijection. 
\end{proposition}

\begin{proof}
Clearly there are no dependencies among the direct summands on the left. Let $(\lambda^1,\hdots,\lambda^{s-1},\nu)\in \mathsf{LR}_r^s$ be arbitrary. Since $\lambda^1_r\le \nu_r$, $(\lambda^1-\lambda^1_r\omega_r,\lambda^2,\hdots,\lambda^{s-1},\nu-\lambda^1_r\omega_r)$ still satisfies (i), nonnegativity, (ii), and (iii). 

But therefore $\lambda^2_r\le \nu_r-\lambda^1_r$, from which we deduce that 
$$(\lambda^1-\lambda^1_r\omega_r,\lambda^2-\lambda^2_r\omega_r,\hdots,\lambda^{s-1},\nu-\lambda^1_r\omega_r-\lambda^2_r\omega_r)$$
once again satisfies (i), nonnegativity, (ii), and (iii). Continuing in this manner we arrive at an element of $\mathsf{C}_{SL_r}^s$, each time subtracting $\lambda^j_rx_j$. 
\end{proof}

\begin{proposition}
Any element of $\mathsf{C}_r^s$ can be written as a sum $z + \sum a_j x_j$, where $z\in \mathsf{LR}_r^s$ and $a_j\in \R$. Likewise, any element of $\mathsf{EqC}_r^s$ can be written as a sum $\hat z + \sum \hat a_j x_j$, where $\hat z\in \mathsf{EqLR}_r^s$ and $\hat a_j\in \R$. 
\end{proposition}

\begin{proof}
Let $x\in \mathsf{C}_r^s$ be arbitrary. If all entries of $x$ are nonnegative, then $x\in \mathsf{LR}_r^s$. Otherwise, $x+B(x_1+x_2+\hdots+x_{s-1})$ will have nonnegative entries for $B\gg 0$, and will still belong to $\mathsf{C}_r^s$, so 
$$
x = \left(x+B(x_1+x_2+\hdots+x_{s-1})\right) - B(x_1+x_2+\hdots+x_{s-1}).
$$
A similar argument works for $\mathsf{EqC}_r^s$, recognizing that for the containment inequalities (iv) will also be satisfied for $B\gg0$. 
\end{proof}

\section{Examples and Counterexamples}\label{x}

\subsection{Using the algorithm}

To illustrate the rays algorithm, let us take a small example where $s=3$, $r=3$. Consider the Horn facet $\hat{\mathcal{F}}$ given by $I_1 = I_2 = \{2\}, K = \{3\}$, with associated equality 
$$
\lambda^1_2+\lambda^2_2 = \nu_3.
$$

\subsubsection{Type I rays}

One choice of type I datum is $j=1,a=2$. Using these, we get $I_1' = \{3\}, I_2' = \{2\}, K' = \{3\}$. The ways to decrement either of the first two sets or increment $K'$ are as follows:
\begin{itemize}
\item $I_1'' = \{2\}, I_2'' = \{2\}, K'' = \{3\} $ $\rightsquigarrow$ $c_{\{2\},\{2\}}^{\{3\}} = 1$ $\implies$ $\lambda^1_2-\lambda^1_3 = 1$. 
\item $I_1'' = \{3\}, I_2'' = \{1\}, K'' = \{3\} $ $\rightsquigarrow$ $c_{\{3\},\{1\}}^{\{3\}} = 1$ $\implies$ $\lambda^2_1-\lambda^2_2 = 1$. 
\item $I_1'' = \{3\}, I_2'' = \{2\}, K'' = \{4\} $ $\rightsquigarrow$ $c_{\{3\},\{2\}}^{\{4\}} = 1$ $\implies$ $\nu_3 = 1$. 
\end{itemize}
All other consecutive differences in $\lambda^1,\lambda^2,\nu$ are $0$ and we get the ray 
$$
(1,1,0;1,0,0;1,1,1).
$$

A second type I datum is $j=2,a=2$. By symmetry, we know from our first calculation that the ray is $(1,0,0; 1,1,0; 1,1,1)$. 

The final type I datum is $j=3, a=3$, resulting in $I_1' = I_2' =K'= \{2\}$. The possible ways to decrement/increment are: 
\begin{itemize}
\item $I_1''=\{1\}, I_2'' = \{2\}, K'' = \{2\}$ $\rightsquigarrow$ $c_{\{1\},\{2\}}^{\{2\}} = 1$ $\implies$ $\lambda^1_1-\lambda^1_2 = 1$. 
\item $I_1''=\{2\}, I_2'' = \{1\}, K'' = \{2\}$ $\rightsquigarrow$ $c_{\{2\},\{1\}}^{\{2\}} = 1$ $\implies$ $\lambda^2_1-\lambda^2_2 = 1$. 
\item $I_1''= \{2\}, I_2'' = \{2\}, K'' = \{3\}$ $\rightsquigarrow$ $c_{\{2\},\{2\}}^{\{3\}} = 1$ $\implies$ $\nu_2-\nu_3 = 1$. 
\end{itemize}
So the ray produced is 
$$
(1, 0, 0; 1, 0, 0; 1,1,0).
$$

\subsubsection{Type II rays}

To find the type II rays on $\hat{\mathcal{F}}_2$, we need to know the rays of $\mathsf{LR}_1^3\times \mathsf{EqLR}_2^3$, which are
\begin{align*}
(1;0;1)\times &(0,0;0,0;0,0) \\
(0;1;1)\times&(0,0;0,0;0,0) 
\end{align*}
together with the 10 rays of the form $(0;0;0)\times z$ where $z$ belongs to the $r=2$ table below. 
For example, let's apply $\wh{\op{Ind}}$ to $(0;0;0)\times (1,0;1,1;1,1)$. First $\pi^{-1}$ takes it to $(1,0,0;1,0,1;1,1,0)$. Then using formula (\ref{tIIform}), $\hat p_2$ sends this to 
$$
(1,0,0;1,0,1;1,1,0) + (1,0,0;1,1,0;1,1,1) - (1,0,0;1,0,0;1,1,0) = (1,0,0;1,1,1;1,1,1),
$$
which is one of the rays in table $r=3$ below. 

For another example, let's apply the induction map to $(1;0;1)\times (0,0;0,0;0,0)$. Applying $\pi^{-1}$ we get $(0,1,0;0,0,0;0,0,1)$. Then $\hat p_2$ takes this to 
$$
(0,1,0;0,0,0;0,0,1) - (1,1,0;1,0,0;1,1,1) + (1,0,0;1,0,0;1,1,0) = 0.
$$
This we expect as noted in the proof of Corollary \ref{noway}. 

The total outputs produced include the following rays:
\begin{align*}
(0,0,0;1,0,0;1,0,0) & & (0,0,0;1,1,1;1,1,1) \\ 
(1,0,0;0,0,0;1,0,0) & & (1,1,1;0,0,0;1,1,1) \\ 
(1,0,0;1,0,0;1,0,0) & & (1,0,0;1,1,1;1,1,1) \\ 
(1,1,1;1,0,0;1,1,1) & & 
\end{align*}
Additionally, $0$ is output three times in accordance with Corollary \ref{noway} (in our example $|\mathfrak{J}| = 3$), and two elements are produced which are not on an extremal ray: 
\begin{align*}
(0;0;0)\times(1,1;1,1;1,1)&\xmapsto{\wh{\op{Ind}}} (2,1,1;2,1,1;2,2,2)\\
(0;0;0)\times(1,1;1,1;2,1)&\xmapsto{\wh{\op{Ind}}} (2,1,1;2,1,1;3,2,2)
\end{align*}

\subsection{Data for small $r$}

To give a sense of the extremal rays of $\mathsf{EqLR}_r^3$, we have recorded the complete list of rays for $r\le 3$. The ``strictly equivariant'' rays (those violating (ii)) are below the dashed line; those which lie on $\mathsf{LR}_r^3$ are above. 

\begin{center}
\def\arraystretch{1.3}
\begin{tabular}[t]{|c|}\hline
$r=1$ \\\hline
(1;0;1) \\
(0;1;1)\\\hdashline
(1;1;1)\\\hline
\end{tabular} 
\quad 
\def\arraystretch{1.3}
\begin{tabular}[t]{|c|}\hline
$r=2$ \\\hline
(0,0;1,0;1,0)\\
 (0,0;1,1;1,1)\\
(1,0;0,0;1,0) \\
(1,1;0,0;1,1) \\
 (1,0;1,0;1,1) \\\hdashline
 (1,0;1,0;1,0) \\
 (1,0;1,1;1,1) \\
 (1,1;1,0;1,1) \\
 (1,1;1,1;1,1) \\
 (1,1;1,1;2,1) \\\hline
\end{tabular}
\quad 
\begin{tabular}[t]{|c|c|c|} \hline
\multicolumn{3}{|c|}{$r=3$} \\\hline
(0,0,0;1,0,0;1,0,0) & (1,0,0;1,0,0;1,1,0) & (1,1,0;1,0,0;1,1,1) \\
(0,0,0;1,1,0;1,1,0) & (1,0,0;1,1,0;1,1,1) & (1,1,0;1,1,0;2,1,1) \\
(0,0,0;1,1,1;1,1,1) & (1,1,0;0,0,0;1,1,0) & (1,1,1;0,0,0;1,1,1) \\
(1,0,0;0,0,0;1,0,0) &  &  \\\hdashline
(1,0,0;1,0,0;1,0,0) & (1,1,0;1,1,0;2,1,0) & (1,1,1;1,1,1;1,1,1) \\
(1,0,0;1,1,0;1,1,0) & (1,1,0;1,1,1;1,1,1) & (1,1,1;1,1,1;2,1,1) \\
(1,0,0;1,1,1;1,1,1) & (1,1,0;1,1,1;2,1,1) & (1,1,1;1,1,1;2,2,1) \\
(1,1,0;1,0,0;1,1,0) & (1,1,1;1,0,0;1,1,1) & (1,1,1;2,1,1;2,2,1) \\
(1,1,0;1,1,0;1,1,0) & (1,1,1;1,1,0;1,1,1) & (2,1,1;1,1,1;2,2,1) \\
(1,1,0;1,1,0;1,1,1) & (1,1,1;1,1,0;2,1,1) &  \\\hline
\end{tabular}
\end{center}~\\

In the following table, we have recorded the number of extremal rays of the cones $\mathsf{LR}_r^3$ and $\mathsf{EqLR}_r^3$ for the first few values of $r$. Calculations were done using {\tt Sage} \cite{Sage} and {\tt Normaliz} \cite{Nmz}. 

\begin{center}
\def\arraystretch{1}
\begin{tabular}{|c|c|c|}\hline 
& &  \\ [-10pt]
$r$ & $\#$ rays of $\mathsf{LR}_r^3$ & $\#$ rays of $\mathsf{EqLR}_r^3$ \\[3pt]\hline
1 & 2 & 3 \\
2 & 5 & 10 \\ 
3 & 10 & 27 \\ 
4 & 20 & 72 \\ 
5 & 44 & 195 \\ 
6 & 114 & 532 \\ 
7 & 362 & 1469 \\\hline
\end{tabular}
\end{center}

\subsection{Extra Hilbert basis elements}

The semigroup of lattice points $\mathsf{EqLR}_r^s\cap \Z^{rs}$ has a finite list of indecomposable elements -- those which are not the sum of two nonzero elements -- called the \emph{Hilbert basis}. Equivalently, the Hilbert basis is the (unique) minimal generating set of $\mathsf{EqLR}_r^s\cap \Z^{rs}$ over $\Z_{\ge0}$. 

Now every extremal ray affords us with exactly one Hilbert basis element, namely the first lattice point along that ray. This element is indecomposable since any summands must be parallel (by extremality) and therefore one of them equal to 0 (by being the first lattice point). 

For general pointed rational cones, the Hilbert basis can be much larger than the set of extremal rays. We observe that for $\mathsf{EqLR}_r^3$, this does not happen until $r=6$; see the table below. Since the natural inclusions $\mathsf{EqLR}_r^s\subset \mathsf{EqLR}_{r+1}^s$ given by appending $0$'s preserve the properties of extremal ray and Hilbert basis element, we conclude that for $r\ge 6$ the Hilbert basis of $\mathsf{EqLR}_r^3\cap \Z^{3r}$ is greater in size than the set of extremal rays. Calculations were once again done using {\tt Sage} and {\tt Normaliz}. ~\\

\begin{center}
\def\arraystretch{1}
\begin{tabular}{|c|c|c|}\hline
& & \\ [-10pt]
$r$ & $\#$ rays of $\mathsf{EqLR}_r^3$ & $\#$ H.b. elts. of $\mathsf{EqLR}_r^3\cap \Z^{3r}$ \\[3pt]\hline 
1 & 3 & 3\\
2 & 10 & 10\\
3 & 27 & 27\\
4 & 72 & 72\\
5 & 195 & 195\\
6 & 532 & {\bf 535}\\
7 & 1469 & {\bf 1500} \\
\hline
\end{tabular}
\end{center}~\\

Here are the three ``extra'' Hilbert basis elements at $r=6$:
\begin{align*}
&(2,1,1,1,1,1;2,2,2,1,1,1;3,3,2,2,2,1)\\
&(2,2,1,1,1,1;2,2,1,1,1,1;3,2,2,2,2,1)\\
&(2,2,2,1,1,1;2,1,1,1,1,1;3,3,2,2,2,1)
\end{align*}
The phenomenon of extra Hilbert basis elements has neither been observed nor ruled out for the cones $\mathsf{LR}_r^3$, having checked the cases $r\le9$ by computer.

\vspace{0.05in}

\noindent
Department of Mathematics, Ohio State University, 231 W 18th Ave, Columbus, OH 43210\\
{{email: kiers.2@osu.edu}}

\end{document}